\documentclass[12pt,reqno]{amsart}
\usepackage{graphicx}
\usepackage{amssymb,amscd, amsthm, latexsym, mathrsfs, epsfig}
\usepackage[all]{xy}
\usepackage{todonotes}

\newcommand\PP{\mathbb P}
\newcommand\A{\mathcal{A}}
\newcommand\C{\mathbb C}

\newcommand\D{\mathcal{D}}
\newcommand\Q{\mathbb Q}
\newcommand\R{\mathbb R}
\newcommand\Z{\mathbb Z}

\newcommand{\olo}{\mathcal{O}}

\renewcommand{\H}{\mathbb{H}}

\newcommand{\eps}{\epsilon}

\newcommand\Aut{\operatorname{Aut}}

\newcommand\dt{\operatorname{DT}}

\newcommand\stab{\operatorname{Stab}}
\newcommand\bra{\langle}
\newcommand\ket{\rangle}

\newcommand{\del}{\partial}

\newcommand{\Hom}{\operatorname{Hom}}

\newcommand{\rk}{\operatorname{rk}}

\newcommand{\Coh}{\operatorname{Coh}}
\newcommand{\Supp}{\operatorname{Supp}}
\newcommand{\tr}{\operatorname{tr}}
\newcommand{\ex}{\operatorname{ex}}
\newcommand{\cl}{\operatorname{cl}}
\newcommand{\ch}{\operatorname{ch}}

\newcommand{\li}{\operatorname{Li}}

\makeatletter \@addtoreset{equation}{section} \makeatother

\newtheorem{thm}{Theorem}
\newtheorem{prop}[thm]{Proposition}
\newtheorem{lem}[thm]{Lemma}
\newtheorem{cor}[thm]{Corollary}

\theoremstyle{definition}
\newtheorem{definition}[thm]{Definition}
\newtheorem{rmk}[thm]{Remark}

\newcommand\narrowdots{\hbox to 1em{.\hss.\hss.}}

\title[Stable pairs, flat connections and GV invariants]{Stable pairs, flat connections and Gopakumar-Vafa invariants}
\date{1 December 2017}
\author{Jacopo Stoppa}
\email{jstoppa@sissa.it}
\address{SISSA, via Bonomea 265, 34136 Trieste, Italy}
\begin{document}

\begin{abstract} Using the interpretation of certain generalised Donald\-son-Thomas invariants, including stable pairs curve counts, as the monodromy of a flat connection on a formal principal bundle, we show that the conjectural Gopakumar-Vafa contributions of all genera to the Gromov-Witten partition function appear in the asymptotics of the corresponding flat sections. The Fourier-Laplace integrals used to produce flat sections lead naturally to the GW/DT change of variable $-q = e^{i u}$.
\end{abstract}

\maketitle

\section{Background and main results}\label{backSec}

\subsection{} Let $X$ be a complex projective Calabi-Yau threefold. The finite rank even homology lattice $\bigoplus^3_{i=0} H_{2i} = \bigoplus^3_{i=0} H_{2i}(X, \Z)/\textrm{tor}$ is endowed with the integral, skew-symmetric, non-degenerate intersection form. We write $(\Gamma, \bra - , - \ket)$ for this lattice and bilinear form. The group algebra $\C[\Gamma]$ can be endowed with its standard commutative product, twisted by the form $\bra - , -\ket$, that is for $x_{\alpha}, x_{\beta} \in \C[\Gamma]$ we set
\begin{equation*}
x_{\alpha} x_{\beta} = (-1)^{\bra \alpha, \beta \ket} x_{\alpha + \beta}, \alpha, \beta \in \Gamma.
\end{equation*}  
Together with the Lie bracket 
\begin{equation*}
[x_{\alpha}, x_{\beta}] = (-1)^{\bra \alpha, \beta \ket} \bra\alpha, \beta\ket x_{\alpha + \beta} 
\end{equation*}
this turns $\C[\Gamma]$ into a Poisson algebra.
\subsection{} Let $\D \subset D^b(\Coh(X))$ be a triangulated subcategory, $U \subset \stab(\D)$ a nonempty open subset of its space of Bridgeland stability conditions. Fixing $\alpha \in \Gamma$, $Z \in U$, we write
\begin{equation*}
\dt(\alpha, Z) \in \Q
\end{equation*} 
for the generalised Donaldson-Thomas enumerative invariant of $Z$-semi\-stable objects in $\D$ whose Chern character (regarded as an element of $\Gamma$, by Poincar\'e duality) equals $\alpha$, assuming this is well defined.
\subsection{} There exist nontrivial examples for the above general setup. In this paper we are concerned with two such examples:
\begin{enumerate} 
\item $\D = \bra \Coh_{\leq 1}(X)\ket_{\tr} \subset D^{b}(\Coh(X))$, the smallest triangulated subcategory containing the torsion coherent sheaves on $X$ supported in dimension $\leq 1$;
\item $\D = \bra \olo_X, \Coh_{\leq 1}(X)\ket_{\tr} \subset D^{b}(\Coh(X))$, where we also allow the structure sheaf of $X$.
\end{enumerate}
Of course (2) generalises (1). It is easy to find stability conditions for $(1)$: fixing the standard heart $\Coh(X)_{\leq 1}$, and writing Chern characters as $(n, \beta) \in H_0 \oplus H_2$, a central charge supported on $\Coh(X)_{\leq 1}$ is given by
\begin{equation*}
Z(n, \beta) = \int_{\beta} \omega_{\C} - n = v_{\beta} - n,
\end{equation*}
where $\omega_{\C} = B + i\omega \in H^2(X, \C)$ is a complexified K\"ahler class. The corresponding DT invariant 
\begin{equation*}
\dt(n, \beta) \in \Q
\end{equation*}
is well-defined and does not depend on $Z$ (\cite{js} Section 6). It is conjecturally related to genus zero Gopakumar-Vafa invariants $n_{0,\gamma}$ in Gromov-Witten theory \cite{gv} by the identity
\begin{equation}\label{torsionConj}
\dt(n, \beta) = \sum_{k > 0, k | (n, \beta)} \frac{n_{0, \beta/k}}{k^2} \in \Z, \text{ for all } n \in \Z 
\end{equation}
(\cite{js} Conjecture 6.20), where the sum is over all positive $k$ dividing the class $(n,\beta) \in H_0\oplus H_2$. 
Example (2) requires tilting: fixing as heart the extension closure 
\begin{equation*}
\A  = \bra \olo_X, \Coh_{\leq 1}(X)[-1]\ket_{\ex}
\end{equation*}
and writing Chern characters as $(-n, -\beta, r) \in H_0 \oplus H_2 \oplus H_6 $, we consider the function  
\begin{equation*}
Z(n, \beta, r) = r G + \int_{\beta}\omega_{\C} - n = r G + v_{\beta} - n
\end{equation*}
for fixed $G \in \H$ with real part $\Re(G) > 0$ (the latter condition is not necessary but simplifies some of our arguments, so we always assume it in the rest of the paper). Toda \cite{todaDTPT} proves that $Z$ is a central charge supported on $\A$. In other words for all classes $[\olo]\in \Gamma$ of nontrivial objects $\olo\in\A$, the complex number $Z([\olo])$ lies in the semi-closed upper half plane $\mathfrak{h} = \H\cup\R_{<0}$. Moreover \cite{todaDTPT, todaHall} the corresponding DT invariant is well defined and, for rank $r = 1$, recovers the stable pairs curve counts $P_{n, \beta}$ of Pandharipande-Thomas \cite{pt}: for all $Z$ as above we have
\begin{equation*}
\dt((-n, -\beta, 1), Z) = P_{n, \beta}.
\end{equation*}
Recall that $P_{n, \beta}$ ``counts" curves in $X$ by virtually enumerating sheaves $F \in \Coh(X)$ with dimension at most $1$, support class $[F] = \beta$ and holomorphic Euler characteristic $\chi(F) = n$, endowed with a generically onto section $s\!: \olo_X \to F$. 

In turn the stable pairs invariants $P_{n,\beta}$ are conjecturally related to (all genera) GV theory by the stable pairs refinement of the GW/DT conjectures (\cite{pt} Section 3.4). Namely let $P'_{n, \beta}$ denote the \emph{connected} stable pairs invariants, introduced in loc. cit. by the identity
\begin{equation*}
\sum_{\beta \neq 0} \sum_{n\in \Z}  P'_{n, \beta} q^{n} x_{\beta} = \log\big(1 + \sum_{\beta\neq0} \sum_{n\in\Z} P_{n,\beta} q^n x_{\beta}\big)
\end{equation*}
(so connected invariants always give a well-defined contribution to $P_{n,\beta}$ and in fact determine the disconnected invariants). Then writing 
\begin{equation*}
F_{P, \beta}(q) = \sum_{n\in \Z}  P'_{n, \beta} q^{n}  
\end{equation*}
for the contribution of degree $\beta$ to connected stable pairs theory, it is expected (\cite{pt} Conjecture 3.14) that there exist unique integers $n_{g, \gamma}$, for every curve class $\gamma$, such that we can express the connected contribution as a finite sum:
\begin{equation}\label{pairsConj}
F_{P,\beta}(q) = \sum_{g \geq 0} \sum_{ r > 0, r | \beta} n_{g,\beta/r} \frac{(-1)^{g-1}}{r}((-q)^r - 2 + (-q)^{-r})^{g-1}.
\end{equation}
We refer to this conjectural identity as \emph{BPS rationality}. In fact integrality and vanishing for sufficiently large $g$ for fixed $\gamma$ are known (\cite{pt} Lemma 3.12 and Theorem 3.20). What remains to be proved in general is vanishing for $g<0$, although this is known in many cases (and always holds when the curve class $\beta$ is irreducible by \cite{ptBPS} Theorem 3). Moreover the connected Gromov-Witten partition function is obtained conjecturally (\cite{pt} Conjecture 3.3) by applying to each $F_{P, \beta}(q)$ the GW/DT change of variable \cite{mnop}
\begin{equation}\label{gwdtChange}
-q = e^{i u},
\end{equation}
and summing over all curve classes $\beta$, recovering the well-known GV expression for the GW partition function \cite{gv}
\begin{equation}\label{gv}
\sum_{g \geq 0} \sum_{\beta} \sum_{r > 0} n_{g,\beta} \frac{1}{r} (2\sin(ru/2))^{2g-2} Q^{r\beta}.
\end{equation}
In this paper we will often assume BPS rationality \eqref{pairsConj}, but we never need the above stable pairs/GV correspondence. It is important to point out that the original GW/DT correspondence has been proved for a vast class of Calabi-Yau threefolds by Pandharipande-Pixton \cite{PandaPixton}.

\subsection{} The original motivation discussed in \cite{mnop} for considering the GW/DT change of variable \eqref{gwdtChange} comes from mathematical physics. One of the aims of the present paper is to show a purely mathematical context in which the change of variable \eqref{gwdtChange} appears naturally within stable pairs theory. Our approach is based on the interpretation of certain generalised Donaldson-Thomas invariants, including stable pairs curve counts, as the monodromy of a flat connection on a formal principal bundle, which we describe in what follows.

Returning to the general setup of $\D \subset D^b(\Coh(X))$ for a moment, let us extend the function $\dt\!: \operatorname{ch}(K(\D)) \to \Q$ by $0$ to all $\Gamma$. Similarly we can allow arbitrary extensions of $Z$ from $\operatorname{ch}(K(\D))$ to all $\Gamma$. For such an extension we fix a ray $\ell \subset \C^*$ and consider the Poisson automorphism of (a suitable completion of) $\C[\Gamma]$ given by
\begin{equation*}
\mathbb{S}_{\ell}(Z) = \exp\big(\big[ \sum_{\alpha \in \Gamma, Z(\alpha) \in \ell} \dt(\alpha, Z) x_{\alpha}, - \big]\big)
\end{equation*}
(\cite{bridRH} Section 2.5, \cite{ks} Section 2.5). The Poisson automorphisms $\mathbb{S}_{\ell}$ are the most effective way to express a fundamental property of generalised DT invariants \cite{js, ks}: fixing a convex sector $V \subset \C^*$, the clockwise ordered product
\begin{equation}\label{wallcross}
\prod^{\to}_{\ell \subset V} \mathbb{S}_{\ell}(Z)
\end{equation}
remains constant in $Z$ as long as no rays $\ell$ with nontrivial $\mathbb{S}_{\ell}$ cross $\del V$ (this requires working in suitable truncations of $\C[\Gamma]$, see \cite{bridRH} Section 3.3, \cite{ks} Section 2.3). This property determines how generalised DT invariants depend on the choice of stability condition (locally described by $Z$). Rigorous definitions of $\mathbb{S}_{\ell}(Z)$ and of the local costancy of \eqref{wallcross} may be found in \cite{ks} and do not play a role in the present paper.
\subsection{} A striking similarity has been noticed between the previous result and the general theory of monodromy for meromorphic connections on a disc with an irregular singularity of Poincar\'e rank $1$ \cite{bt_stab, gmn, joy}. Following the approach developed in \cite{AnnaJ, fgs}, but ignoring the details of completion\footnote{In the apprach of \cite{AnnaJ, fgs}, fixing a generic $Z$, one introduces a formal parameter $s$ and an inverse system of principal bundles $P^j = \Aut(\C[\Gamma][s]/(s)^j) \times \C^*$, with connections $\nabla^j$, whose generalised monodromies converge to a formal power series version of $\{\mathbb{S}_{\ell}(Z), \ell \subset \C^*\}$, and are constant in a possibly decreasing sequence of open neighbourhoods $U_j$ of $Z$ in $\Hom(\Gamma, \C)$. But in general, without extra finiteness assumptions, we have $\cap_{j} U_j= \emptyset$.}, the insight is that one should regard the collection $\mathbb{S}_{\ell}(Z) \in \Aut(\C[\Gamma]), \ell \subset \C^*$ (with $\mathbb{S}_{\ell} \neq 1$) for fixed $Z$ as the generalised monodromy (``Stokes factors") around $t = 0$ of a meromorphic connection $\nabla(Z)$ on the trivial $\Aut(\C[\Gamma])$-bundle over a disc, with only a double pole at $t = 0$. Then the local constancy of $\prod^{\to}_{\ell \subset V} \mathbb{S}_{\ell}(Z)$ becomes precisely the statement that this generalised monodromy is constant, i.e. the family of connections $\nabla(Z)$ parametrised by $Z$ is isomonodromic. In fact the simplest situation one could consider takes place on $\C^* \subset \PP^1$ rather than a disc and is given by a family of connections of the form
\begin{equation*}  
\nabla(Z) = d - \left(\frac{Z}{t^2} + \frac{f(Z)}{t}\right)dt.
\end{equation*}
Here $Z$ is regarded as a diagonal element of $\operatorname{Der}(\C[\Gamma])$ acting by
\begin{equation*}
Z(x_{\alpha}) = Z(\alpha)x_{\alpha},
\end{equation*}
and $f(Z)$ takes values in $\operatorname{Der}(\C[\Gamma])$.


\subsection{} The connections $\nabla$ are unique and can be constructed rigorously \cite{AnnaJ, bt_stab, fgs}. But one could ask what is gained from taking this point of view. In order to show this, although it is not a priori clear why doing so should be useful, let us consider the (unique) normalised flat sections $Y(t, Z) \in \Aut(\C[\Gamma])$ of $\nabla(Z)$, i.e. their ``canonical solutions" (here we continue to suppress the details of completion\footnote{In the approach of \cite{AnnaJ, fgs}, $Y(t, Z)$ takes values in $\Aut(\C[\Gamma][[s]])$, and its reductions modulo $(s)^j$ give canonical sections for the inverse system of connections $\nabla^j$.}). They can be written explicitly in terms of a sum over graphs $T$, with vertices $v$ decorated by elements $\alpha_v \in \Gamma$, of the form
\begin{equation*}
Y(t, Z)(x_{\alpha}) = x_{\alpha} \exp\big( t^{-1} Z(\alpha) + \bra\alpha, \sum_{T} W_{T}(Z) H_{T}(t, Z)\ket\big).
\end{equation*}
Here the weights $W_{T}(Z)$ are monomials in $\C[\Gamma] \otimes_{\Z} \Gamma_{\Q}$, given essentially by the product of $\dt(\alpha_v) x_{\alpha_v}$ over vertices, while the functions $H_T(t, Z)$ are complicated iterated integrals (Fourier-Laplace transforms of ``periods"), holomorphic in $t$ with possible branch cuts\footnote{The appearence of branch cuts is expected: around a double pole, the ``canonical solutions" of a holomorphic differential equation are only defined in a sector.} along rays $\ell \subset \C^*$ with $\mathbb{S}_{\ell} \neq 1$. Explicit formulae are known. The sum over graphs is infinite and must be interpreted as a formal power series in one or more auxiliary formal parameters \cite{AnnaJ, fgs}. This aspect is irrelevant for the present paper: assuming BPS rationality \eqref{pairsConj}, we will prove a stronger convergence statement for the particular infinite sums which appear in our main result.
\subsection{} Returning to the example (1) of $\D = \bra \Coh_{\leq 1}(X)\ket_{\tr} \subset D^{b}(\Coh(X))$ we observe that it possesses a very special property: the locus in $\Gamma$ where $\dt$ does not vanish is ``Lagrangian", i.e. the restriction of the form $\bra - , - \ket$ to this locus is the zero form. Indeed $\dt$ vanishes except for classes $(n, \beta) \in H_{0} \oplus H_{2}$, which pair to zero under the intersection form of the threefold $X$. In this ideal case, in the expansion for a flat section $Y(t, Z)$ as a sum over graphs $T$, the only contributions come from single-vertex $T$ decorated by all possible classes $(n, \beta) \in H_0 \oplus H_2$,
\begin{equation*}
T = \overset{(n, \beta)}{\bullet}.
\end{equation*}  
We write $W_{(n, \beta)}(Z)$, $H_{(n, \beta)}(t, Z)$ for the monomial and function corresponding to $T$. Now fix a 
primitive curve class $\beta$. The contribution to flat sections $Y(t, Z)$ of $\nabla(Z)$ from $\beta$ and all its multiples via the ``potential"
\begin{equation*}
\sum_{T} W_{T}(Z) H_{T}(t, Z) 
\end{equation*} 
is given by
$\sum_{k, n \in \Z} W_{ k(n, \beta)}(Z) H_{ k(n, \beta)}(t, Z)$, 
corresponding to the graphs
\begin{equation}\label{genus0graphs}
T = \overset{k (n, \beta)}{\bullet}.
\end{equation}
We can turn off the dependence on the variables of $\C[\Gamma]$ by specialising each basic monomial $x_{\alpha} \mapsto 1$ (the idea of considering this specialisation first appeared in \cite{gaiotto} and is developed in \cite{bridRH}). We denote this specialisation by
\begin{equation}\label{p0Def}
p_{0,\beta}(t, Z) = \sum_{k,n \in \Z} W_{k(n, \beta)}(Z)|_{x_{\bullet} = 1} H_{k(n, \beta)}(t, Z) \in H_0 \oplus H_2.
\end{equation}
We will see that in fact $p_{0, \beta}(t, Z)$ is well defined as a formal power series in the variables $t$, $e^{2\pi i v_{\beta}}$. A key computation of Bridgeland and Iwaki, based on results of Bridgeland, shows that the potentials $p_{0,\beta}(t, Z)$ essentially recover the genus $0$ GV contribution (the $g =0$ term in \eqref{gv}), up to a simple partial differential equation. Let $\beta^{\vee}$ be the homology class dual to $\beta \in \Gamma$.
\begin{thm}[Bridgeland, Bridgeland-Iwaki \cite{bridRH} Section 6.3]\label{bridThm} Let $\beta$ be a primitive curve class. Assume the torsion sheaves/GV conjectural identity \eqref{torsionConj}, and suppose $n_{0, k\beta}$ vanishes\footnote{We assume primitivity and the vanishing in order to simplify the statement. A similar result holds in general.} for $k>0$.  
Then the formal power series in $t$, $e^{2\pi i v_{\beta}}$ given by
\begin{equation*}
-\frac{1}{2\pi i}\del_t  \bra \beta^{\vee}, p_{0,\beta}(t, Z) \ket
\end{equation*}
and the Laurent series in $t$ 
\begin{equation*}
 n_{0, \beta} \del_{v_{\beta}} \sum_{r > 0} \frac{1}{r}(2\sin(r ((2\pi)^2 t))/2))^{-2} e^{2\pi i r v_{\beta}},
\end{equation*}
are well defined and agree except for two terms of order $t^{-2}$, $t^0$.
\end{thm}
Setting $Q^{\beta} = e^{2\pi i\omega_{\C} \cdot \beta} = e^{2\pi iv_{\beta}}$ is precisely how GV contributions are written in the physics literature \cite{gv}. In any case Theorem \ref{bridThm} shows how the conjectural genus $g = 0$ GV contribution, thought of as a formal power series, can be computed from the asymptotics of  the flat sections of $\nabla$ as $t \to 0$. 
We hope that this result clearly shows an interesting use\footnote{A relation between Theorem \ref{bridThm} and classical ODEs is discussed in \cite{jj}. A different type of relation between the connection $\nabla$ and curve-counting in described in \cite{fgs, fs}.} of the connection $\nabla$ and of its flat sections. 
\subsection{} The principal aim of this paper is to describe an analogue of Theorem \ref{bridThm}, which is also valid for arbitrary genus. The GW/DT change of variable \eqref{gwdtChange} will emerge naturally in the process and will be closely related to a Fourier-Laplace transform.

For this we turn to Toda's construction 
(2), i.e. the case of $\D = \bra \olo_X, \Coh_{\leq 1}(X)\ket_{\tr} \subset D^{b}(\Coh(X))$. Unfortunately in this case the non-vanishing locus of $\dt$ inside $\Gamma$ is no longer ``Lagrangian": for example, for all $(-n, -\beta, 1)\in H_0\oplus H_2 \oplus H_6$ we have $\dt(-n, -\beta, 1) = P_{n, \beta}$, and 
\begin{equation*}
\bra (-n, -\beta, 1), (-n', -\beta', 1)\ket = n - n'. 
\end{equation*}
Because of this the contribution of a curve class $\beta$ to flat sections of $\nabla$ via the potential
$\sum_{T} W_{T}(Z) H_{T}(t, Z)$ 
is much more involved. The corresponding graphs $T$ may contain an arbitrary number of edges, for example of the form 
\begin{equation*} 
\overset{(-n, - \beta,1)}{\bullet} \longrightarrow \overset{(-n', 0,0)}{\bullet}.
\end{equation*}
Establishing an analogue of Theorem 1 requires choosing a suitable subset of all the possible contributions and relating this to GV theory. We will show that this can be achieved by looking at the contribution of all ``framed" single-vertex graphs
\begin{equation}\label{genusGgraph}
T = \overset{\pm(-n, - \beta,1)}{\bullet} \longrightarrow \overset{\pm[\olo_X]}{\bullet},
\end{equation} 
(for all choices of signs), where we write $[\olo_X] \in \Gamma$ for the class of the structure sheaf. While the closest analogue of the genus $0$ case is given by the single-vertex graphs
\begin{equation}\label{genusGUnframedGraph}
T = \overset{\pm(-n, - \beta,1)}{\bullet},  
\end{equation}
for genus $g > 1$ this choice leads to a serious difficulty with signs. Framing by $\olo_X$ fixes the wrong signs.

\subsection{}\label{vanishingSec} Our restriction to summing only over the graphs \eqref{genusGgraph} loses a lot of the information contained in $\nabla$, but it is not arbitrary. 

Firstly it makes sense to restrict to decorated graphs $T$ such that the decoration at each vertex $v$ is of the form $\pm (-n_v, -\beta_v, r_v)$ for positive rank $r_v > 0$. This is because we can sum over all rank $0$ decorations, using Theorem \ref{bridThm}, and simply obtain copies of the genus $0$ GV contributions of various curve classes $\beta_v$. 

Secondly we can consider the asymptotic behaviour of the contribution of such a positive rank graph $T$, in the limit when we rescale $G = Z(\olo_X)$ to $\lambda G$, for $\lambda > 0$, and let $\lambda \to \infty$ (Toda \cite{todaConifold} equation (1) considers this rescaling and calls it a ``large volume" limit). We show in Lemma \ref{vanishingLem} that the contribution of a positive rank graph $T$ with $m$ vertices dacays at least as 
\begin{equation*}
\big(\prod_v \frac{1}{\sqrt{2 r_v}} \big)\lambda^{-m/2}
\end{equation*}
when $\lambda \to \infty$, up to a factor independent of $\lambda$ and the $r_v$. Thus, among all positive rank graphs $T$, the leading contribution beyond that of the single-vertex graphs comes from those of the form \eqref{genusGgraph}.

While in the present paper we restrict to rank $1$, it may be interesting to try and extend our analysis to all positive ranks. Note that Toda \cite{todaHall} studies higher rank analogues of stable pairs invariants and proves a structure result for their generating function.  
\subsection{} We can now describe our main result. Fix a curve class $\beta$. Following the genus $0$ case, let us write 
\begin{equation*}
W_{(-n, - \beta,1), [\olo_X]}(Z), H_{(-n,- \beta,1), [\olo_X]}(t, Z)
\end{equation*}
for the monomial and function corresponding to the framed single-vertex graph \eqref{genusGgraph}. We will see that the weight $W_{(-n, - \beta,1), [\olo_X]}$ is essentially the same as the invariant $\dt(-n, - \beta,1) = P_{n, \beta}$ (in particular it is constant for our choices of $Z$), so there is a well defined \emph{connected} contribution $W'_{(-n, -r\beta,1), [\olo_X]}$ corresponding to the connected invariant $P'_{n, \beta}$. Then the higher genus analogue of $p_0(t, Z)$ is given by
\begin{align*}
&p_{\beta}(t, Z)\\
&= \sum_{n \in \Z}  \sum_{\pm} W'_{\pm(-n, - \beta,1), [\olo_X]}|_{x_{\bullet} = 1} H_{\pm(-n, - \beta,1), [\olo_X]}(t, Z)\\& \in H_0 \oplus H_2 \oplus H_6,
\end{align*} 
where as before we specialise each basic monomial $x_{\alpha} \mapsto 1$. Assuming BPS rationality \eqref{pairsConj} we will see that there is a natural, finite genus expansion for $p_{\beta}(t, Z)$. 
The genus $0$ term is a divergent integral, but admits a canonical choice of regularisation, while each $g > 0$ term is in fact a well defined \emph{function} of $t \in \R$. We are interested in the behaviour of these functions near $t = \infty$, where we regard $t$ as a real variable, so we will use a real coordinate $\epsilon$ with $\epsilon = t^{-1}$. Introduce a differential operator 
\begin{equation*}
\mathcal{L} = -i \del_{v_{\beta}} \eps^{-1} (\del_{\eps} - \eps^{-1}).
\end{equation*}
Our main result shows that the conjectural genus $g > 0$ Gopakumar-Vafa contribution of each curve class $\beta$ to the Gromov-Witten partition function, which according to \eqref{gv} is given by
\begin{equation*}
\sum_{r > 0, r|\beta} n_{g,\beta/r} \frac{1}{r} (2\sin(ru/2))^{2g-2} e^{i v_{\beta}},
\end{equation*}
can be reconstructed from the asymptotics of flat sections of $\nabla$ as $t \to \infty$. We do not actually need to assume the stable pairs/GV conjecture, but only BPS rationality \eqref{pairsConj}. 
\begin{thm}\label{mainThm} Assume BPS rationality \eqref{pairsConj}. Regard $t = \eps^{-1}$ as a real variable. Then for each curve class $\beta$ there is a natural, finite genus expansion
\begin{equation*}
\lim_{G \to 0} \bra \beta^{\vee}, p_{\beta}(t, Z)\ket =  \hat{p}_{0, \beta}(t, Z) + \sum_{g > 0} p_{g, \beta}(t, Z),
\end{equation*}
where $\hat{p}_{0, \beta}(t, Z)$ is a divergent integral, admitting a canonical choice of regularisation, while each genus $g > 0$ term 
$p_{g, \beta}(t, Z)$
is a well defined \emph{function} of $t = \eps^{-1}$, such that for all $j \geq 1$ we have
\begin{align*}
&\lim_{\eps\to 0}\Big(\mathcal{L}^{j}  p_{g, \beta}(\eps^{-1}, Z)\\
&- \sum^{\lfloor (j-1)/2 \rfloor}_{h = 0}  (-1)^{j+h}\frac{(2h)!}{\eps^{2h+1}} \del^{j-2h-1}_{u}\big(e^{iu v_{\beta} } i\del_{u}\sum_{r>0, r|\beta} n_{g, \beta/r}\frac{1}{r}(2\sin(ru)/2)^{2g-2} \big)\big)|_{u = 0} \Big)\\
&=  - \frac{(-1)^j}{4}  \del^{j}_{u}\big( e^{iu v_{\beta}} i\del_{u}\sum_{r>0, r|\beta} n_{g, \beta/r}\frac{1}{r}(2\sin(ru)/2)^{2g-2} \big)|_{u = 0}.
\end{align*}
\end{thm}
It follows immediately that all the coefficients of the Taylor expansion of each conjectural genus $g > 0$ GV contribution of $\beta$ at $u = 0$ are determined by the asymptotics of flat sections of $\nabla$ as $t \to \infty$ (by induction on powers $\mathcal{L}^j$). Note that here we exclude the (singular) genus $0$ case. Moreover the result gives no information on the genus $g = 1$ GV contribution, since that is constant in $u$. In fact the approach of Theorem \ref{mainThm} can be extended to both cases. For genus $0$ see Remark \ref{genus0Rmk}. Recovering the genus $g = 1$ contribution involves the ``unframed" potential corresponding to the graphs \eqref{genusGUnframedGraph}, and is discussed in detail after the proof of Theorem \ref{mainThm}. 

Comparing the proofs of Theorems \ref{bridThm} and \ref{mainThm}, we find that the conjectural torsion sheaves/GV identity \eqref{torsionConj} is replaced by BPS rationality \eqref{pairsConj}. But in fact the two conjectures are entirely equivalent, as shown by Toda (\cite{todaKyoto} Theorem 6.4).  A key role in the proof of Theorem \ref{mainThm} is played by properties of the Poisson kernel $\kappa_{\eps}(u)$, i.e. of the integral kernel for the Laplacian on the upper half plane with Dirichlet boundary conditions; we express each genus $g > 0$ contribution $p_g(\eps^{-1}, Z)$ through convolution with $\kappa_{\eps}(u)$.

\subsection{} The promised interpretation of the GW/DT change of variable $-q = e^{iu}$ emerges naturally in the course of the proof Theorem \ref{mainThm} (see Remark \ref{dtgwRmk}). The essential point is that, when taking flat sections of the connection $\nabla$, the contribution of an invariant $P'_{n, \beta}$ appears as the coefficient of a Fourier-Laplace integral with respect to   
\begin{equation*}
e^{i \sigma Z(-n, -\beta, 0, 1)} d\sigma = e^{i\sigma(G +  v_{\beta} - n)} d\sigma, 
\end{equation*}
involving the Poisson kernel $\kappa_{1/t}(\sigma)$. Such Fourier-Laplace integrals arise classically when taking flat sections around a double pole, as we are doing. The integral is then localised around $\sigma = 0$ in the limit $t \to \infty$ (as the Poisson kernel $\kappa_{1/t}(\sigma)$ concentrates at $0$), and the phase $e^{- i n \sigma}$ replaces $(- q)^n$.  

\subsection{} In Section \ref{generalSec} we briefly recall some general results on the connection $\nabla$. Section \ref{genus0Sec} offers a proof of Theorem \ref{bridThm} which is closer to our approach in the higher genus case. Toda's construction is reviewed in Section \ref{todaSec}. Our main result Theorem \ref{mainThm} is proved in Section \ref{genusGSec} (the special case of genus $g = 1$ is discussed in detail at the end of the Section). 

Finally we note that recent works of Maulik and Toda \cite{mt, todaGV} suggest the possibility of obtaining a higher genus analogue of Theorem \ref{bridThm} working with the category $\D = \bra \Coh_{\leq 1}(X)\ket_{\tr} \subset D^{b}(\Coh(X))$, but considering instead a deformation of the connection $\nabla$, of a type described in \cite{fs} Section 4.1. It might be interesting to develop this approach and to compare it to the one of the present paper, based on stable pairs.\\

\noindent\textbf{Acknowledgements.}  I am very grateful to Anna Barbieri, Tom Bridgeland, Jacopo Scalise, and Richard Thomas for important comments and suggestions about this work. The research leading to these results has received funding from the European Research Council under the European Union's Seventh Framework Programme (FP7/2007-2013) / ERC Grant agreement no. 307119.
\section{Some basic results on $\nabla$}\label{generalSec}
  
In this section we briefly discuss some general results on the connection $\nabla$ and its flat sections, following the approach developed in \cite{fgs} (see also \cite{AnnaJ, fs, j}), but ignoring the details of completion. They are not relevant for the purposes of this paper. We assume for definiteness we are in the setup (1) or (2) described in Section \ref{backSec}, although the results do not depend on the details of these cases. Importantly, we work with the ``symmetric" version of $\nabla$ described in \cite{fgs} sections 2.4, 3.11, 4.6, which is obtained by imposing the relation $\dt(\alpha, Z) = \dt(-\alpha, Z)$ (geometrically, this symmetry comes from the shift functor $[1] \in \Aut(\D)$). 

Recall $\nabla(Z)$ is a meromorphic connection on the trivial $\Aut(\C[\Gamma])$-bundle\footnote{In general, as we explained, one should work with the inverse system of bundles and connections $(P^i, \nabla^i)$, see \cite{AnnaJ, fgs}.} over $\PP^1$, $\Aut(\C[\Gamma])$ denoting the automorphism group of $\C[\Gamma]$ as a commutative, associative algebra endowed with the twisted product (\cite{fgs} Section 4). More explicitly we have
\begin{equation*}
\nabla(Z) = d - \left(\frac{Z}{t^2} + \frac{f(Z)}{t} \right)dt
\end{equation*}
where $Z \in \Hom(\Gamma, \C)$ is regarded as a diagonal element of the derivation module $\operatorname{Der}(\C[\Gamma])$, and $f(Z) \in \operatorname{Der}(\C[\Gamma])$. So $\nabla$ has an irregular singularity at $t = 0$, and a regular singularity at $t = \infty$. Because of the irregular singularity at $t = 0$, it is well known that in general it is not possible to find a flat section of $\nabla$ in a whole neighbourhood of $t = 0$, even after passing to a suitable ramified cover. The obstruction comes from the generalised monodromy around $t=0$, that is from rays (``Stokes rays") 
\begin{equation*}
\ell = \R_{ > 0} Z(\alpha) \subset \C^*
\end{equation*}
such that the generalised monodromy automorphism (``Stokes factor") $\mathbb{S}_{\ell}(Z)$ is nontrivial (\cite{fgs} Section 4.5). Fixing an open convex sector $\Sigma$ between two such rays\footnote{Rays of this type could be dense in a region, and in general the statement only makes sense for each $(P^i, \nabla^i)$.}, one can show the existence of a unique flat section which is asymptotic to
$\exp(t^{-1}Z)$
as $t \to 0$ in $\Sigma$ (\cite{fgs} Section 4.3). This solution has polynomial growth as $t \to \infty$ in $\Sigma$. 

In the present case such flat sections can be written explicitly. They are given by the restriction to sectors $\Sigma$ of an $\Aut(\C[\Gamma])$-valued function of the form
\begin{equation}\label{flatSections}
Y(t, Z)(x_{\alpha}) = x_{\alpha} \exp\big( t^{-1} Z(\alpha) + \bra\alpha, \sum_{T} W_{T}(Z) H_{T}(t, Z)\ket\big) 
\end{equation}
(\cite{fgs} sections 3.6, 4.3). Here we sum over oriented graphs $T$ whose vertices are decorated by elements of $\Gamma$. For fixed $Z$, each $H_{T}(t, Z)$ is a holomorphic function with branch cuts in the variable $t\in \C^*$, whose restrictions to sectors $\Sigma$ are holomorphic, while the weights $W_T(Z)$ lie in $\C[\Gamma]\otimes_{\Z}\Gamma_{\Q}$. The general formulae for $W_T(Z)$, $H_T(t, Z)$ are quite involved, but as explained in Section \ref{backSec} we are only interested in the precise expression for the initial (``linear and quadratic") terms in the sum over graphs, corresponding to 
\begin{align*}
T(\alpha) &= \overset{\alpha}{\bullet},\quad
T(\alpha, \alpha') = \overset{\alpha}{\bullet} \longrightarrow \overset{\alpha'}{\bullet}. 
\end{align*}
Let us write $W_{\alpha}(Z)$, $H_{\alpha}(t, Z)$, respectively $W_{\alpha, \alpha'}(Z)$, $H_{\alpha, \alpha'}(t, Z)$ for the corresponding weights and functions. 
\begin{lem} We have
\begin{align*}
W_{\alpha}(Z) &= \dt(\alpha, Z) \alpha \otimes x_{\alpha},\\
H_{\alpha}(Z) &= \frac{1}{2\pi i}\int_{\R_{>0}Z(\alpha)}\frac{dz}{z}\frac{t}{z-t} e^{- Z(\alpha)/z},
\end{align*}
respectively
\begin{align*}
W_{\alpha,\alpha'}(Z) &= \dt(\alpha, Z) \dt(\alpha',Z)\bra\alpha,\alpha'\ket \alpha \otimes x_{\alpha} x_{\alpha'}\\
&= \dt(\alpha, Z)\dt(\alpha',Z) (-1)^{\bra\alpha, \alpha'\ket} \bra\alpha,\alpha'\ket \alpha \otimes x_{\alpha + \alpha'},\\
H_{\alpha,\alpha'}(t, Z) &= \frac{1}{2\pi i}\int_{\R_{>0}Z(\alpha)}\frac{d z}{z}\frac{t}{z-t} e^{- Z(\alpha)/z} H_{\alpha'}(z, Z).
\end{align*}
\end{lem}
\begin{proof} This is a special case of the general formulae \cite{fgs} equations (4.8), (4.9); the weights $W_T$ are described in ibid. Section 3.6.
\end{proof}
We can turn off the dependence on the variables of $\C[\Gamma]$ by specialising each basic monomial $x_{\alpha} \mapsto 1$. We are interested in the behaviour of the above contributions under this specialisation. Given $\alpha \in \Gamma$ let us write
\begin{align*}
\widehat{H}_{\alpha}(t, Z) &= \frac{1}{2\pi i}\int_{\R_{>0}Z(\alpha)}dz\frac{2t}{z^2-t^2} e^{- Z(\alpha)/z},\\
\widehat{H}_{\alpha,\alpha'}(t, Z) &= \frac{1}{2\pi i}\int_{\R_{>0}Z(\alpha)}dz\frac{2t}{z^2-t^2} e^{- Z(\alpha)/z} \widehat{H}_{\alpha'}(z,Z).
\end{align*}
\begin{lem}\label{conformalLem} For fixed $\alpha, \alpha' \in \Gamma$, we have
\begin{equation*}
\sum_{\pm }W_{\pm\alpha}(Z)|_{x_{\bullet} = 1} H_{\pm\alpha}(t, Z) = - W_{\alpha}(Z)|_{x_{\bullet}=1}\widehat{H}_{\alpha}(t, Z),
\end{equation*}
\end{lem}
respectively
\begin{align*}
\sum_{(\pm, \pm)}W_{\pm\alpha, \pm\alpha'}(Z)|_{x_{\bullet} = 1} H_{\pm\alpha, \pm\alpha'}(t, Z) = W_{\alpha,\alpha'}(Z) \widehat{H}_{\alpha, \alpha'}(t,Z)
\end{align*}
(sum over all pairs of signs).
\begin{proof} This is a straightforward computation using the property
\begin{equation*}
\dt(\alpha, Z) = \dt(-\alpha,Z)
\end{equation*}
due to the shift functor $[1] \in \Aut(\D)$. 
\end{proof}
The last general result we mention concerns the case of ``Lagrangian" $\D \subset D^b(\Coh(X))$.
\begin{lem}\label{lagrangianLem} Suppose the restriction of the form $\bra - , - \ket$ to the locus in $\Gamma$ where $\dt\!: \Gamma \to \Q$ does not vanish is the zero form. Then we have 
\begin{equation*}
\sum_{T} W_{T}(Z) H_{T}(t, Z) = \sum_{\alpha \in \Gamma} W_{\alpha}(Z) H_{\alpha}(t, Z),
\end{equation*}
and in particular
\begin{equation*}
\sum_{T} W_{T}(Z)|_{x_{\bullet} = 1} H_{T}(t, Z) = -\sum_{\alpha \in \Gamma, Z(\alpha) \in \mathfrak{h}} W_{\alpha}(Z) \widehat{H}_{\alpha}(t, Z).
\end{equation*}
\end{lem} 
\begin{proof} Suppose the graph $T$ has vertices $v$ decorated by $\alpha(v) \in \Gamma$, and write $v \to w$ for an edge. According to the general formula for a weight $W_{T}(Z)$ given in \cite{fgs} Section 3.6, $W_{T}(Z)$ contains a factor of the form
\begin{equation*}
\prod_{v \in T} \dt(\alpha(v), Z) \prod_{v \to w} \bra \alpha(v), \alpha(w)\ket.
\end{equation*}
The first claim follows at once. The second claim then follows from Lemma \ref{conformalLem}.
\end{proof}
\section{Relation to genus $0$ Gopakumar-Vafa invariants}\label{genus0Sec}

This Section contains a different proof of Theorem \ref{bridThm} which is closer to our treatment of higher genera. 

Recall we are concerned with the ``Lagrangian" case of torsion sheaves, i.e. the category $\D = \bra \Coh_{\leq 1}(X)\ket_{\tr} \subset D^{b}(\Coh(X))$. Fix a primitive curve class $\beta \in H_2$. This contributes with all its multiples to flat sections of $\nabla$ through all the homology classes $k(n, \beta) \in H_0 \oplus H_2$, for $ k, n \in \Z$. 
\begin{lem}\label{genus0Lem} The contribution of all multiples of the primitive curve class $\beta \in H_2$ to flat sections of $\nabla$ via $\sum_T W_{T}(Z) H_{T}(t, Z)$ is given by
\begin{equation*}
\sum_{k, n \in \Z}W_{k(n, \beta)} H_{k(n, \beta)}(t, Z). 
\end{equation*}
Let $p_{0,\beta}(t,Z)$ denote its specialisation under $x_{\alpha} \mapsto 1$. Assume the conjectural torsion sheaves/GV identity \eqref{torsionConj}, and suppose $n_{0, k\beta}$ vanishes for $k>0$. 
Then we have an equality of formal power series in $t$, $e^{2\pi i v_{\beta}}$
\begin{align*}
 \bra \beta^{\vee}, p_{0,\beta}(t, Z)\ket =  n_{0, \beta} \sum_{k>0}\frac{1}{ \pi k } \sum_{n > 0} \frac{  \frac{2\pi}{k} t}{1 + (\frac{2\pi}{k} n t)^2} e^{  2\pi i n v_{\beta}}. 
\end{align*}
\end{lem} 
\begin{proof} The vanishing required by Lemma \ref{lagrangianLem} holds, so the first claim follows from the first equality in that Lemma. Using the second equality in Lemma \ref{lagrangianLem} and the special form of the central charge
\begin{equation*}
Z(k(n, \beta)) = kv_{\beta} - kn = k\beta\cdot(B + i\omega) - kn
\end{equation*}
gives formally
\begin{align*}
&p_{0,\beta}(t, Z) \\&= \sum_{k,n \in \Z, Z(k(n,\beta)) \in \mathfrak{h}} W_{ k(n, \beta)}|_{x_{\bullet}=1} \widehat{H}_{k(n,\beta)}(t,Z)\\
&=  - \sum_{n\in\Z}\sum_{k>0} k(n, \beta) \dt(k(n, \beta)) \frac{1}{2\pi i} \int_{\R_{>0}(v_{\beta} - n)} dz \frac{2 t}{z^2 - t^2} e^{-k(v_{\beta} - n)/z}.
\end{align*}
The following argument will show in particular that the right hand side is well defined as a formal power series in $t$, $e^{2\pi iv_{\beta}}$. We may assume $t\in\R$. A straightforward application of the residue theorem shows that we can move the integration contour to the pure imaginary line, i.e. setting $z = i\sigma$, $\sigma \in \R_{>0}$ we have 
\begin{equation*}
p_{0,\beta}(t, Z) = \sum_{n\in\Z}\sum_{k>0} k(n, \beta) \dt(k(n, \beta)) \frac{1}{2\pi} \int^{\infty}_{0} d\sigma \frac{2 t}{\sigma^2 + t^2} e^{i k(v_{\beta} - n)/\sigma}.
\end{equation*}
Assuming the conjectural identity \eqref{torsionConj}, primitivity of $\beta$ and the vanishing $n_{0, k\beta} = 0$ for $k >1$ give
\begin{equation*}
\dt(k(n, \beta)) = \frac{n_{0, \beta}}{k^2} 
\end{equation*}
for all $k > 0$, $n \in \Z$, and combining this identity with the change of variable $\sigma \mapsto \sigma^{-1}$ gives
\begin{equation*}
\bra \beta^{\vee}, p_{0,\beta}(t, Z)\ket = n_{0, \beta}\sum_{k > 0} \frac{1}{2\pi k}\sum_{n\in\Z}    \int^{\infty}_{0} d\sigma \frac{2 t}{1 + (\sigma t)^2} e^{i \sigma(v_{\beta} - n)}.
\end{equation*}
We use the well known distributional identity
\begin{equation*}
\frac{k}{2\pi}\sum_{n \in \Z} e^{ i \sigma k n} = \sum_{n \in \Z} \delta(\sigma - \frac{2\pi}{k} n)
\end{equation*}
in order to perform the sum over $n \in \Z$. Integrating the resulting Dirac comb on $(0, \infty)$ gives the required identity.
\end{proof}
\begin{proof}[Proof of Theorem \ref{bridThm}] 
Set 
$\lambda = 2\pi t$
in the expression found in Lemma \ref{genus0Lem} and expand in powers of $\lambda$, giving
\begin{align*}
& \bra \beta^{\vee}, p_{0}(t, Z)\ket\\
&= n_{0, \beta} \sum_{k>0}\frac{1}{ \pi k} \sum_{n > 0} \sum_{p \geq 0} (-1)^p n^{2p} \left(\frac{\lambda}{k}\right)^{2p+1} e^{  2\pi i n v_{\beta}}\\
&= n_{0, \beta} \sum_{k>0}\frac{1}{\pi } \sum_{n > 0} \sum_{p \geq 0}\frac{1}{ k^{2p+2}} (-1)^p n^{2p} \lambda^{2p+1} e^{  2\pi i n v_{\beta}}\\
&= n_{0, \beta} \sum_{k>0}\frac{1}{\pi } \sum_{p \geq 0} \frac{1}{ k^{2p+2}}(-1)^p \li_{-2p}(e^{2\pi i v_{\beta}}) \lambda^{2p+1}. 
\end{align*}
We can sum over $k$ using the well-known expression for the zeta function at positive even integers in terms of Bernoulli numbers,
\begin{equation*}
\sum_{k>0}\frac{1}{k^{2p+2}} = \zeta(2(p+1)) = -\frac{(-1)^{p+1}B_{2(p+1)}(2\pi)^{2(p+1)}}{2(2(p+1))!}, 
\end{equation*}
yielding
\begin{align*}
& \bra \beta^{\vee}, p_{0}(t, Z)\ket\\
&=- n_{0, \beta} \frac{1}{\pi }\sum_{p \geq 0}\frac{(-1)^{p+1}B_{2(p+1)}(2\pi)^{2(p+1)}}{2(2(p+1))!}  (-1)^p \li_{-2p}(e^{2\pi i v_{\beta}})\lambda^{2p+1}.
\end{align*}
Setting 
$p = g - 1$
we can write this as
\begin{align*}
& \bra \beta^{\vee}, p_{0}(t, Z)\ket\\
&= -  n_{0, \beta}  \frac{1}{\pi}\sum_{g \geq 1}\frac{(-1)^{g}B_{2g}(2\pi)^{2g}}{2(2g)!}  (-1)^{g-1} \li_{2-2g}(e^{2\pi i v_{\beta}})\lambda^{2g-1}\\
&=  n_{0, \beta}\sum_{g \geq 1}\frac{B_{2g}}{(2g)!} \li_{2-2g}(e^{2\pi i v_{\beta}})(2\pi \lambda)^{2g-1}.
\end{align*}
As a consequence
\begin{align*}
& \del_t \bra \beta^{\vee}, p_{0}(t, Z)\ket \\
&=  n_{0, \beta}(4\pi)^2\sum_{g \geq 1}\frac{B_{2g}}{(2g)(2g-2)!} (2\pi \lambda)^{2g-2} \li_{2-2g}(e^{2\pi i v_{\beta}}).
\end{align*}
We compare this to the standard expansion of the genus $0$ GV contribution
\begin{align*}
& n_{0, \beta} \sum_{r > 0} \frac{1}{r}(2\sin(r u/2))^{-2} Q^{r \beta}\\
&= b_0 u^{-2} + b_1 +  n_{0,\beta} \sum_{g \geq 2} \frac{(-1)^{g-1} B_{2g}}{(2g)(2g-2)!} u^{2g-2}\li_{3-2g}(Q^{\beta})\\
&=b_0 (2\pi\lambda)^{-2} + b_1 +  n_{0,\beta} \sum_{g \geq 2} \frac{(-1)^{g-1} B_{2g}}{(2g)(2g-2)!} (2\pi \lambda)^{2g-2}\li_{3-2g}(e^{2\pi iv_{\beta}}),
\end{align*}
where in the last equality we evaluate at $u = 2\pi \lambda$, $Q^{\beta} = e^{2\pi iv_{\beta}}$. The identity
\begin{equation*}
\del_{z} \li_{s+1}(z) = z^{-1} \li_{s}(z), 
\end{equation*}
implies
\begin{equation*}
\del_{z} \li_{s+1}(e^{2\pi i z}) = 2\pi i \li_{s}(e^{2\pi i z}),
\end{equation*}
so the derivative with respect to $v_{\beta}$ of the positive degree part of the genus $0$ GV contribution equals
\begin{equation*}
(2\pi i)  n_{0, \beta}\sum_{g \geq 2}\frac{B_{2g}}{(2g)(2g-2)!} (2\pi\lambda)^{2g-2}\li_{2-2g}(e^{2\pi i v_{\beta}}).
\end{equation*}
By our computations this agrees with 
$-\frac{1}{2\pi i}  \del_t \bra \beta^{\vee}, p_{0}(t, Z)\ket$,
for positive powers of $\lambda$.
\end{proof}
 
\section{Toda's construction}\label{todaSec}

In this section we follow closely Toda \cite{todaDTPT} Section 3. We set
\begin{equation*}
\Coh_{\leq 1}(X) = \{ E \in \Coh(X)\!: \dim\Supp(E) \leq 1\}
\end{equation*}
and introduce the triangulated category
\begin{equation*}
\D = \bra \olo_X, \Coh_{\leq 1}(X)\ket_{\tr} \subset D^{b}(\Coh(X)).
\end{equation*}
The right orthogonal to $\Coh_{\leq 1}(X)$ is defined as
\begin{equation*}
\Coh_{\geq 2}(X) = \{E \in \Coh(X)\!: \Hom(\Coh_{\leq 1}(X), E) = 0\}
\end{equation*}
and the pair of subcategories of $\Coh(X)$ given by 
\begin{equation*}
(\Coh_{\leq 1}(X), \Coh_{\geq 2}(X))
\end{equation*}
forms a torsion pair, with respect to which we can tilt $\Coh(X)$ inside $D^b(X)$, i.e. consider the extension closure
\begin{equation*}
\Coh^{\dagger}(X) = \bra \Coh_{\geq 2}(X)[1], \Coh_{\leq 1}(X) \ket_{\ex}.
\end{equation*}
In this way we obtain a new \rm{t}-structure on $D^b(\Coh(X))$ and, by restriction, on $\D$. Indeed \cite{todaDTPT} Lemma 3.5 shows that the intersection
\begin{equation*}
\A_X = \D \cap \Coh^{\dagger}(X)[-1]
\end{equation*} 
in $D^b(\Coh(X))$ defines the heart of a bounded \rm{t}-structure on $\D$. The result also shows that this abelian category can be written more explicitly as
\begin{equation*}
\A = \bra \olo_X, \Coh_{\leq 1}(X)[-1]\ket_{\ex}.
\end{equation*} 
Toda then proceeds to construct both classical and weak stability conditions on $\D$. These are described in terms of the sublattice 
\begin{equation*}
\Gamma' = H_0 \oplus N_{1}(X) \oplus H_6 \subset \Gamma,
\end{equation*}
where $N_1(X)$ denotes the abelian group of curves in $X$. Similarly we write $N^1(X)$ for the group of divisors, with the perfect intersection pairing 
\begin{equation*}
N_1(X)_{\R} \times N^1(X)_{\R} \ni (C, D) \to C \cdot D.
\end{equation*}
We write elements of $\Gamma'$ in the form 
\begin{equation*}
(s, l, r) \in \Z \oplus N_{1}(X) \oplus \Z, 
\end{equation*}
so the last component corresponds to the rank $\rk\!:\Gamma \to \Z$. There is a group homomorphism $\cl\!:K(\D) \to \Gamma' \subset \Gamma$ given by the Chern character, 
\begin{equation*}
\cl(E) = (\ch_3(E), \ch_2(E), \ch_0(E)).
\end{equation*}
\begin{prop}[Toda \cite{todaDTPT} Lemma 3.8 and Remark 3.16]\label{todaProp} Fix
\begin{equation*}
A \in \R_{> 0},\quad B + i\omega \in N^1(X)_{\C}, \quad G \in \H
\end{equation*}   
and define a group homomorphism 
\begin{equation*}
Z\!: \Gamma \ni (s, l, r) \mapsto A s - (B + i\omega)\cdot l + rG \in \C.
\end{equation*}
Then $(Z, \A)$ is a classical stability condition on $\D$, framed by $\Gamma'$, i.e. a point of $\stab_{\Gamma'}(\D)$.  
\end{prop}
\begin{thm} [Toda \cite{todaDTPT} Theorem 3.13 and \cite{todaHall}]\label{todaThm} In the situation of Proposition \ref{todaProp} there is a well-defined generalised Donalson-Thomas invariant $\dt\!:\Gamma'\to\Q$ enumerating $Z$-semistable objects of $\D$. Moreover for all such $Z$ we have an equality with Pandharipande-Thomas stable pairs invariants
\begin{equation*}
\dt(-n, -\beta, 1) = P_{n, \beta}.
\end{equation*} 
\end{thm}
In the rest of the paper we set $A = 1$ (as in Section \ref{backSec}) and we assume the condition $\Re(G) > 0$. Finally we prove the simple vanishing result mentioned in Section \ref{vanishingSec}.
\begin{lem}\label{vanishingLem} In the situation of Proposition \ref{todaProp}, let $T$ be a graph with $m$ vertices in the expression \eqref{flatSections} for flat sections of $\nabla$. Suppose that the decoration at each vertex $v$ is of the form $\alpha_v = \pm (-n_v, -\beta_v, r_v)$ for positive rank $r_v > 0$. Consider the scaling $G \mapsto \lambda G$ for $\lambda >0$. Then the contribution of $T$ to flat sections of $\nabla(\lambda) = \nabla(\lambda G, v_{\beta})$ decays at least as $\big(\prod_{v}\frac{1}{\sqrt{2 r_v}}\big)\lambda^{-m/2}$ as $\lambda \to \infty$, up to a factor independent of $\lambda$ and the $r_v$.
\end{lem}
\begin{proof} According to the general formulae \cite{fgs} equations (4.8), (4.9), the weight function $H_T(\lambda G, v_{\beta})$ is given by an iterated integral, whose integrand contains a factor of the form
\begin{equation*}
\frac{1}{2\pi i}\int_{\R_{>0}Z(\alpha_v)}\frac{dz_v}{z_v}\frac{z_u}{z_v - z_u} e^{- Z(\alpha_v)/z_v}
\end{equation*}
for each edge $u \to v$. By the Cauchy-Schwarz inequality and our assumptions this is bounded in modulus by
\begin{equation*}
\left(\int^{\infty}_0 e^{2 r_v \lambda \Re(i G) \sigma} d\sigma\right)^{\frac{1}{2}} = \frac{1}{\sqrt{- 2 r_v \Re(i G) \lambda}} 
\end{equation*}
up to a factor independent of $\lambda$ and $r_v$ (note that we have $\Re(i G) < 0$). The claim follows by applying this bound to each vertex $v$.
\end{proof}
\section{Relation to higher genus Gopakumar-Vafa invariants}\label{genusGSec}
In this Section we prove Theorem \ref{mainThm}. We work with Toda's construction, i.e. in the ``non-Lagrangian" case of $\D = \bra \olo_X, \Coh_{\leq 1}(X)\ket_{\tr} \subset D^{b}(\Coh(X))$. We always assume BPS rationality \eqref{pairsConj}.

We are concerned with the contribution of a curve class $\beta$ to the flat sections of $\nabla$ through the potential
\begin{equation*}
\sum_{T} W_{T}(Z) H_{T}(t, Z).
\end{equation*}
This sum is highly complicated, but contains a distinguished contribution which is a very close analogue of the genus $0$ potential $p_0(t, Z)$, namely the sum over single-vertex graphs
\begin{equation*}
\sum_{n \in \Z}  \sum_{\pm} W_{\pm(-n, -\beta,1)}(Z)|_{x_{\bullet}} H_{\pm(-n, -\beta,1)}(t, Z).
\end{equation*}
Unfortunately it turns out that working directly with this analogue leads to a serious difficulty with signs. We resolve this by working instead with the ``framed" potential
\begin{align*}
\sum_{n \in \Z} \sum_{(\pm,\pm)} W_{\pm(-n, -\beta,1), \pm[\olo_X]}(Z)|_{x_{\bullet} = 1} H_{\pm(-n, -\beta,1),\pm[\olo_X]}(t, Z).
\end{align*}
\begin{lem}\label{framedPotLem} For fixed $n \in \Z$, regarding $t$ as a real variable, we have
\begin{align*}
&\sum_{(\pm,\pm)} W_{\pm(-n, - \beta,1), \pm[\olo_X]}|_{x_{\bullet} = 1}(Z) H_{\pm(-n, - \beta,1),\pm[\olo_X]}(t, Z)\\
&= -(-1)^{n} n P_{n,   \beta}(-n, - \beta, 1)\\
&\quad \int^{\infty}_0 d\sigma \frac{1}{\pi} \frac{t}{1 + (\sigma t)^2} e^{i\sigma(G +  v_{\beta} - n)} \int^{\infty}_0d\tau \frac{1}{\pi} \frac{\sigma}{1+ (\tau \sigma)^2}e^{- G/\tau}.
\end{align*}
\end{lem}
\begin{proof} According to the second identity in Lemma \ref{conformalLem} we have
\begin{align*}
&\sum_{(\pm,\pm)} W_{\pm(-n, - \beta,1), \pm[\olo_X]}(Z)|_{x_{\bullet} = 1} H_{\pm(-n, - \beta,1),\pm[\olo_X]}(t, Z)\\
&= W_{(-n, - \beta,1), [\olo_X]}|_{x_{\bullet} = 1} \widehat{H}_{(-n, - \beta,1),[\olo_X]}(t, Z).
\end{align*}
Recall
\begin{align*}
&W_{(-n, - \beta,1), [\olo_X]}(Z)\\
&= \dt(-n, - \beta,1) \dt([\olo_X]) \bra (-n, - \beta,1), [\olo_X]\ket (-n, - \beta, 1)\\ &\quad\otimes x_{(-n, - \beta,1)} x_{[\olo_X]},
\end{align*}
so 
\begin{equation*}
W_{(-n, - \beta,1), [\olo_X]}(Z)|_{x_{\bullet = 1}} = (-1)^n n P_{n,  \beta} (-n, - \beta,1).
\end{equation*}
On the other hand we have
\begin{align*}
&H_{(-n, - \beta,1),[\olo_X]}(t, Z)\\
&= \frac{1}{2\pi i}\int_{\R_{>0} (G +  v_{\beta} - n)} dz \frac{2t}{z^2-t^2} e^{-(G +  v_{\beta} - n)/z}\frac{1}{2\pi i }\int_{\R_{> 0} G} dw \frac{2z}{w^2 - z^2} e^{-G/w}.
\end{align*}
As we are assuming $t \in \R$, arguing as in the proof of Lemma \ref{genus0Lem}, we can integrate instead over $z = i \sigma$, $\sigma \in \R$, so
\begin{align*}
&H_{(-n, - \beta,1),[\olo_X]}(t, Z)\\
&= -\frac{1}{2\pi}\int^{\infty}_{0} d\sigma \frac{2t}{\sigma^2 + t^2} e^{i(G +  v_{\beta} - n)/\sigma}\frac{1}{2\pi }\int_{\R_{> 0} G} dw \frac{2\sigma}{w^2 + \sigma^2} e^{-G/w}.
\end{align*}
By the same argument we can integrate instead over $w = \tau \in \R$, 
\begin{align*}
&H_{(-n, - \beta,1),[\olo_X]}(t, Z)\\
&= -\int^{\infty}_{0} d\sigma \frac{1}{\pi}\frac{t}{\sigma^2 + t^2} e^{i(G +  v_{\beta} - n)/\sigma}\int^{\infty}_{0} d\tau \frac{1}{\pi}\frac{\sigma}{\tau^2 + \sigma^2} e^{-G/\tau}.
\end{align*}
The claim follows from the change of integration variable $\sigma \mapsto \sigma^{-1}$.
\end{proof}
As explained in Section \ref{backSec}, we focus on the \emph{connected} potential, corresponding to the connected invariants $P'_{n, \beta}$.
\begin{definition} The framed, connected single-vertex contribution of a curve class $\beta$ to flat sections of $\nabla$ is given by 
\begin{align*}
& p_{\beta}(t, Z) = \sum_{n \in \Z} \sum_{(\pm,\pm)} W'_{\pm(-n, -\beta,1), \pm[\olo_X]}(Z)|_{x_{\bullet} = 1} H_{\pm(-n, -\beta,1),\pm[\olo_X]}(t, Z)
\end{align*}
where
\begin{align*}
&\sum_{(\pm,\pm)}W'_{\pm(-n, -\beta,1), \pm[\olo_X]}(Z)|_{x_{\bullet} = 1} H_{\pm(-n, -\beta,1),\pm[\olo_X]}(t, Z)\\
&= -(-1)^{n} n P'_{n,   \beta}(-n, - \beta, 1)\\
&\quad \int^{\infty}_0 d\sigma \frac{1}{\pi} \frac{t}{1 + (\sigma t)^2} e^{i\sigma(G +  v_{\beta} - n)} \int^{\infty}_0d\tau \frac{1}{\pi} \frac{\sigma}{1+ (\tau \sigma)^2}e^{- G/\tau}.
\end{align*}
\end{definition}
\begin{prop}\label{genusGProp} We have
\begin{align*}
&\lim_{G\to 0} \bra \beta^{\vee}, p_{\beta}(t, Z)\ket\\
&= \sum_{g \geq 0} \sum_{r>0, r|\beta}n_{g, \beta/r}   \int^{\infty}_0 d\sigma \frac{1}{2\pi}\frac{t}{1+(\sigma t)^2} e^{i   \sigma  v_{\beta}}i\del_{\sigma} \frac{1}{r}(2\sin(r\sigma/2))^{2g-2} 
\end{align*}
where the $g = 0$ term is a divergent integral, admitting a canonical choice of regularisation, while each $g > 0$ term is a well defined \emph{function} of $t\in \R$.
\end{prop}
\begin{proof} By the previous Lemma, noting that for all $\sigma \in \R$
\begin{equation*}
\int^{\infty}_0d\tau \frac{1}{\pi} \frac{\sigma}{1+ (\tau \sigma)^2}  = \frac{1}{2}
\end{equation*}
we have
\begin{align*}
&\lim_{G\to 0} \bra \beta^{\vee}, p_{\beta}(t, Z)\ket =  \sum_{n \in\Z} (-1)^n n P'_{n,  \beta} \int^{\infty}_0 d\sigma \frac{1}{2\pi} \frac{t}{1 + (\sigma t)^2} e^{i\sigma( v_{\beta} - n)}.  
\end{align*}
According to  \eqref{pairsConj} we have\footnote{The operator $[q^n]$ applied to Laurent polynomials extracts the coefficient of $q^n$.}
\begin{align*}
P'_{n, \beta} = \sum_{g\geq 0} \sum_{r>0, r|\beta}n_{g, \beta/r} [q^{n}]\left(\frac{(-1)^{g-1}}{r}((-q)^{r} - 2 + (-q)^{-r})^{g-1}\right), 
\end{align*}
a finite sum over $g$, so we can rewrite
\begin{align*}
&\lim_{G\to 0} \bra \beta^{\vee}, p_{\beta}(t, Z)\ket\\
&= \sum_{g \geq 0} \sum_{r>0, r|\beta}n_{g, \beta/r}  \int^{\infty}_0 d\sigma \frac{1}{2\pi}\frac{t}{1+(\sigma t)^2} e^{i \sigma v_{\beta}}\\
& \quad\,\sum_{n\in \Z} [(-q)^n] \left(\frac{(-1)^{g-1}}{r}((-q)^{r} - 2 + (-q)^{-r})^{g-1}\right)  n e^{- i n \sigma}.
\end{align*}
We note that
\begin{align*}
&\sum_{n\in \Z} [(-q)^n] \left(\frac{(-1)^{g-1}}{r}((-q)^{r} - 2 + (-q)^{-r})^{g-1}\right)  n e^{- i n \sigma}\\
&= i\del_{\sigma}\left(\left(\frac{(-1)^{g-1}}{r}((-q)^{r} - 2 + (-q)^{-r})^{g-1}\right)\big|_{-q = e^{-i\sigma} }\right)\\
&=i\del_{\sigma} \frac{(-1)^{g-1}}{r}(2(\cos(-r\sigma) - 1))^{g-1}\\ 
&= i\del_{\sigma} \frac{1}{r}(2\sin(r\sigma/2))^{2g-2}.
\end{align*}
This proves the identity claimed by the Proposition. It is straightforward to see that each $g>0$ summand is integrable for all $t \in \R$ and so gives a well defined function of $t$. 

The $g = 0$ term is not integrable near the boundary $\{\sigma = 0\}$, but becomes so after multiplication by a sufficiently large power of the defining function $\sigma$. It is straightforward to check that the smallest such power is $\sigma^3$. In this situation one has a canonical choice of regularisation of the divergent integral, its ``Hadamard finite part" (see e.g. \cite{felder} Section 1). In our case it is given by
\begin{equation*}
\sum_{r>0, r|\beta}n_{0, \beta/r}   \int^{\infty}_0 d\sigma \frac{1}{2\pi}\frac{t}{1+(\sigma t)^2} e^{i   \sigma  v_{\beta}} \left(i\del_{\sigma}\frac{1}{r}(2\sin(r\sigma/2))^{-2} - \frac{2}{r^3 \sigma^3}\right).
\end{equation*}

\end{proof}
Recall that one defines the Poisson kernel for the upper half plane as the family of functions
\begin{equation*}
\kappa_{\eps}(\sigma) = \frac{1}{\pi}\frac{\eps}{\eps^2 + \sigma^2} \in C^{\infty}(\R).  
\end{equation*}
parametrised by $\eps \in \R_{> 0}$. It is well known that this in the integral kernel for the Laplace operator with Dirichlet boundary conditions on the upper half plane, i.e. given $f \in L^2(\R)$, the convolution
\begin{equation*}
u(x + i y) = \int_{\R} k_{y}(x - t) f(t) dt
\end{equation*}
is well defined and harmonic for $y > 0$, with $u(x + i y) \to f(x)$ in $L^2$. In particular as $\eps \to 0$ we have 
\begin{equation*}
\kappa_{\eps}(\sigma) \to \delta(\sigma) \in (\mathcal{S}(\R))'.
\end{equation*}
Moreover the kernel $k_{\eps}(\sigma)$ is symmetric in $\sigma$, and satisfies interesting partial differential equations in $\eps$, $\sigma$. A particular first order equation is especially useful for our purposes.
\begin{lem} We have
\begin{equation}\label{basicPDE}
( \del_{\eps} - \eps \sigma^{-1}\del_{\sigma} - \eps^{-1}) \kappa_{\eps}(\sigma) = 0.
\end{equation}
\end{lem}
\begin{proof} A straightforward computation.
\end{proof}
\begin{lem}\label{genusExpansionLem} For $t = \eps^{-1}\in\R$ we have
\begin{equation*}
\lim_{G\to 0} \bra \beta^{\vee}, p_{\beta}(t, Z)\ket = \hat{p}_{0,\beta}(t, Z) + \sum_{g \geq 1} p_{g, \beta}(t)
\end{equation*}
where the $g = 0$ term is a divergent integral, admitting the canonical choice of regularisation 
\begin{align*}
&\hat{p}^{\rm{reg}}_{0,\beta}(\eps^{-1}, Z)\\ &= \sum_{r>0, r|\beta} n_{0, \beta/r} \frac{1}{2}\sum_{r>0} \int^{\infty}_0 d\sigma \kappa_{\eps}(\sigma) e^{i   \sigma v_{\beta}} \left(i\del_{\sigma} \frac{1}{r}(2\sin(r\sigma/2))^{-2}-\frac{2}{r^3\sigma^3}\right),
\end{align*}
while for $g > 0$ we have well defined \emph{functions} of $t \in \R$,
\begin{align*}
p_{g,\beta}(\eps^{-1}, Z) &= \sum_{r>0, r|\beta} n_{g, \beta/r} \frac{1}{2}\sum_{r>0} \int^{\infty}_0 d\sigma \kappa_{\eps}(\sigma) e^{i \sigma v_{\beta}} i\del_{\sigma} \frac{1}{r}(2\sin(r\sigma/2))^{2g-2}.
\end{align*}
\end{lem}
\begin{proof} This follows at once from Proposition \ref{genusGProp}. 
\end{proof}
We recognise the genus $g$ contribution $p_{g,\beta}(\eps^{-1}, Z)$ as a multiple of the convolution of the functions 
$k_{\eps}(\sigma), \, e^{i  \sigma v_{\beta}} i\del_{\sigma} \frac{1}{r}(2\sin(r\sigma/2))^{2g-2}\chi_{[0, \infty)}$,
evaluated at $0$.
Introduce a differential operator
\begin{equation*}
\mathcal{L} = -i\del_{v_{\beta}} \eps^{-1} (\del_{\eps} - \eps^{-1}).
\end{equation*}
\begin{cor}\label{DiffOp} For $t\in\R, t=\eps^{-1}$ and all $g > 0, j \geq 1$ we have
\begin{align*}
&\mathcal{L}^{j} p_{g,\beta}(\eps^{-1}, Z)\\&=  \sum_{r>0, r|\beta} n_{g, \beta/r} \frac{1}{2} \int^{\infty}_0 d\sigma \del^j_{\sigma}(\kappa_{\eps}(\sigma)) e^{i\sigma  v_{\beta}}\big(i\del_{\sigma}\frac{1}{r}(2\sin(r\sigma)/2)^{2g-2}\big).
\end{align*}
\end{cor}
\begin{proof}
First we compute using \eqref{basicPDE} and Lemma \ref{genusExpansionLem}
\begin{align*}
&\del_{\eps}  p_{g, \beta}(\eps^{-1}, Z)\\
& = \sum_{r>0, r|\beta} n_{g, \beta/r} \frac{1}{2}\sum_{r>0} \int^{\infty}_0 d\sigma\del_{\eps}\kappa_{\eps}(\sigma) e^{i\sigma v_{\beta}}i\del_{\sigma}\frac{1}{r}(2\sin(r\sigma)/2)^{2g-2}\\
&= \sum_{r>0, r|\beta} n_{g, \beta/r} \\&\frac{1}{2} \int^{\infty}_0 d\sigma(\eps\sigma^{-1}\del_{\sigma}\kappa_{\eps}(\sigma) + \eps^{-1}\kappa_{\eps}(\sigma)) e^{i\sigma  v_{\beta}}i\del_{\sigma}\frac{1}{r}(2\sin(r\sigma)/2)^{2g-2}.
\end{align*}
This shows that we have
\begin{align*}
&-i\del_{v_{\beta}} \eps^{-1} (\del_{\eps} - \eps^{-1}) p_{g, \beta}(\eps^{-1}, Z)\\
& = \sum_{r>0, r|\beta}n_{g, \beta/r} \frac{1}{2} \int^{\infty}_0 d\sigma(\del_{\sigma}\kappa_{\eps}(\sigma))  e^{i\sigma  v_{\beta}}i\del_{\sigma}\frac{1}{r}(2\sin(r\sigma)/2)^{2g-2}
\end{align*}
which is the claim for $j = 1$. Assuming the claim for $\mathcal{L}^{j-1}$, we compute
\begin{align*}
&\del_{\eps} \mathcal{L}^{j-1} p_{g,\beta}(\eps^{-1}, Z)\\ 
&=\sum_{r>0, r|\beta}n_{g,\beta/r}\frac{1}{2} \int^{\infty}_0 d\sigma \del_{\eps} \del^{j-1}_{\sigma} \kappa_{\eps}(\sigma) e^{i\sigma v_{\beta} } \big( i\del_{\sigma}\frac{1}{r}(2\sin(r\sigma)/2)^{2g-2}\big)\\
&=\frac{1}{2} \sum_{r>0, r|\beta}n_{g,\beta/r} \\&\int^{\infty}_0 d\sigma \del^{j-1}_{\sigma} (\eps\sigma^{-1}\del_{\sigma}\kappa_{\eps}(\sigma) + \eps^{-1}\kappa_{\eps}(\sigma)) e^{i\sigma v_{\beta} } \big( \del_{\sigma}\frac{1}{r}(2\sin(r\sigma)/2)^{2g-2}\big).
\end{align*}
It follows that 
\begin{align*}
&-i\del_{v_{\beta}}\eps^{-1}(\del_{\eps} - \eps^{-1})\mathcal{L}^{j-1} p_{g}(\eps^{-1}, Z)\\
&= \sum_{r>0, r|\beta} n_{g, \beta/r} \frac{1}{2}\sum_{r>0} \int^{\infty}_0 d\sigma(\del^j_{\sigma}\kappa_{\eps}(\sigma)) e^{i\sigma r v_{\beta}}\big(i\del_{\sigma}\frac{1}{r}(2\sin(r\sigma)/2)^{2g-2}\big) 
\end{align*}
which is the claim for $\mathcal{L}^j$.
\end{proof}
\begin{proof}[Proof of Theorem \ref{mainThm}] From Corollary \ref{DiffOp}, integrating by parts, we obtain
\begin{align}\label{byParts}
\nonumber&\mathcal{L}^j  p_{g, \beta}(\eps^{-1},Z)\\
\nonumber&= \frac{1}{2} \sum_{r>0, r|\beta} n_{g,\beta/r}  \left( [\del^{j-1}_{\sigma}(\kappa_{\eps}(\sigma))  e^{i\sigma  v_{\beta}} \big(i\del_{\sigma}\frac{1}{r}(2\sin(r\sigma)/2)^{2g-2}\big)]^{\sigma = \infty}_{\sigma = 0}\right.\\
&\left. - \int^{\infty}_0 d\sigma \del^{j-1}_{\sigma}(\kappa_{\eps}(\sigma)) \del_{\sigma}\big(  e^{i\sigma  v_{\beta}} \big(i\del_{\sigma}\frac{1}{r}(2\sin(r\sigma)/2)^{2g-2}\big)\big)\right).
\end{align} 
By vanishing as $\sigma \to \infty$ we have
\begin{align*}
&[\del^{j-1}_{\sigma}(\kappa_{\eps}(\sigma))  e^{i\sigma  v_{\beta}} \big(i\del_{\sigma}\frac{1}{r}(2\sin(r\sigma)/2)^{2g-2}\big)]^{\sigma = \infty}_{\sigma = 0}\\
& = - \big(\del^{j-1}_{\sigma}(\kappa_{\eps}(\sigma))  e^{i\sigma  v_{\beta}} \big(i\del_{\sigma}\frac{1}{r}(2\sin(r\sigma)/2)^{2g-2}\big)\big)|_{\sigma = 0}.
\end{align*}
So integrating by parts $j$ times we find
\begin{align*}
&\mathcal{L}^j  p_{g, \beta}(\eps^{-1},Z)= \frac{1}{2} \sum_{r>0, r|\beta} n_{g,\beta/r}\\
& \Big(-\sum^{j}_{k = 1}(-1)^{j-k}\big(\del^{k-1}_{\sigma}(\kappa_{\eps}(\sigma)) \del^{j-k}_{\sigma}\big( e^{i\sigma  v_{\beta} }  i\del_{\sigma}\frac{1}{r}(2\sin(r\sigma)/2)^{2g-2} \big)\big)|_{\sigma = 0} \\
&  - (-1)^j \int^{\infty}_0 d\sigma\kappa_{\eps}(\sigma) \del^{j}_{\sigma}\big(e^{i\sigma  v_{\beta} }  i\del_{\sigma}\frac{1}{r}(2\sin(r\sigma)/2)^{2g-2} \big)\Big).
\end{align*}
Now we use the expansion around $\sigma = 0$
\begin{equation*}
\kappa_{\eps}(\sigma) = \sum^{\infty}_{h = 0} 2\frac{(-1)^h}{\eps^{2h+1}} \sigma^{2h}
\end{equation*}
to compute for $h\geq 0$
\begin{align*}
\del^{2h}_{\sigma}\kappa_{\eps}(\sigma)|_{\sigma=0} = 2\frac{(-1)^h}{\eps^{2h+1}}(2h)!,\quad \del^{2h+1}_{\sigma}\kappa_{\eps}(\sigma)|_{\sigma=0} = 0.
\end{align*}
Substituting in \eqref{byParts} gives
\begin{align*}
&\mathcal{L}^{j}  p_{g, \beta}(\eps^{-1}, Z)=\frac{1}{2} \sum_{r>0, r|\beta} n_{g, \beta/r}\\
& \Big( 2(-1)^j\sum^{\lfloor (j-1)/2 \rfloor}_{h = 0}  (-1)^h\frac{(2h)!}{\eps^{2h+1}} \del^{j-2h-1}_{\sigma}\big(e^{i\sigma v_{\beta} } i\del_{\sigma}\frac{1}{r}(2\sin(r\sigma)/2)^{2g-2} \big)\big)|_{\sigma = 0} \\
&  - (-1)^j \int^{\infty}_0 d\sigma\kappa_{\eps}(\sigma) \del^{j}_{\sigma}\big( e^{i\sigma v_{\beta}} i\del_{\sigma}\frac{1}{r}(2\sin(r\sigma)/2)^{2g-2} \big)\Big).
\end{align*}
Recall that we have $\kappa_{\eps}(\sigma) \to \delta(\sigma)$ in $(\mathcal{S}(\R))'$ as $\eps \to 0$, and moreover $\kappa_{\eps}(\sigma) = \kappa_{\eps}(-\sigma)$. In particular we have
\begin{align*}
&\lim_{\eps\to0}\int^{\infty}_0 d\sigma\kappa_{\eps}(\sigma) \del^{j}_{\sigma}\big( e^{i\sigma  v_{\beta} } i\del_{\sigma}\frac{1}{r}(2\sin(r\sigma)/2)^{2g-2} \big)\\
&=\frac{1}{2} \del^{j}_{\sigma}\big( e^{i\sigma v_{\beta} } i\del_{\sigma}\frac{1}{r}(2\sin(r\sigma)/2)^{2g-2} \big)|_{\sigma = 0}. 
\end{align*}
The claim follows.
\end{proof}
\begin{rmk}\label{genus0Rmk} Clearly the proof extends to the regularised genus $0$ term $\hat{p}^{\rm{reg}}_{0,\beta}(\eps^{-1}, Z)$; we omit the details.
\end{rmk}
\begin{rmk}\label{dtgwRmk} The change of variable $-q = e^{-i\sigma}$ emerges naturally from the proof of the Proposition \ref{genusGProp}. Theorem \ref{mainThm} shows that it is the same as the GW/DT change of variable $-q = e^{iu}$.
\end{rmk}
Finally we discuss how to recover the genus $g = 1$ GV contribution
$\sum_{r>0, r | \beta} n_{1, \beta/r}\frac{1}{r} e^{i v_{\beta}}$. 
Consider the unframed, connected potential 
\begin{align*}
\tilde{p}(t, Z) &= \sum_{n \in \Z} \sum_{\pm} W'_{\pm(-n, -\beta,1)}(Z)|_{x_{\bullet}} H_{\pm(-n, -\beta,1)}(t, Z)\\
&= \sum_{n \in \Z} \sum_{\pm} \pm(-n,-\beta, 1) P'_{n, \beta} H_{\pm(-n, -\beta,1)}(t, Z)
\end{align*}
Regarding $t$ as a real variable, and arguing as in Lemma \ref{framedPotLem} we find
\begin{align*}
&\sum_{\pm} W'_{\pm(-n, - \beta,1)}|_{x_{\bullet} = 1}(Z) H_{\pm(-n, - \beta,1)}(t, Z)\\
&= P'_{n, \beta}(-n, - \beta, 1) \int^{\infty}_0 d\sigma \frac{1}{\pi} \frac{t}{1 + (\sigma t)^2} e^{i\sigma(G +  v_{\beta} - n)}.
\end{align*}
(note in particular a missing factor of $(-1)^n n$ since we are now looking at single-vertex unframed graphs). Then proceeding as in Proposition \ref{genusGProp} we see
\begin{align*}
\lim_{G\to 0} \bra \beta^{\vee}, \tilde{p}(t, Z)\ket &= \sum_{g \geq 0} \sum_{r>0, r|\beta} n_{g, \beta/r} \int^{\infty}_0 d\sigma \frac{1}{2\pi}\frac{t}{1+(\sigma t)^2} e^{i  \sigma v_{\beta}}\\
& \quad\,\sum_{n\in \Z} [q^n] \left(\frac{(-1)^{g-1}}{r}((-q)^{r} - 2 + (-q)^{-r})^{g-1}\right)  e^{- i n \sigma},
\end{align*}
crucially missing the $(-1)^n n$ factor. Note that
\begin{align*}
&\sum_{n\in \Z} [q^n] \left(\frac{(-1)^{g-1}}{r}((-q)^{r} - 2 + (-q)^{-r})^{g-1}\right)  e^{- i n \sigma}\\
&=  \left(\left(\frac{(-1)^{g-1}}{r}((-q)^{r} - 2 + (-q)^{-r})^{g-1}\right)\big|_{-q = e^{-i(\sigma-\pi)} }\right)\\
&= \frac{(-1)^{g-1}}{r}(2(\cos(-r(\sigma-\pi)) - 1))^{g-1}\\ 
&=  \frac{1}{r}(2\sin(r(\sigma-\pi)/2))^{2g-2}.
\end{align*}
Then as in Lemma \ref{genusExpansionLem} we can write the genus expansion
\begin{equation*}
\lim_{G\to 0} \bra \beta^{\vee}, \tilde{p}(t, Z)\ket = \sum_{g \geq 0} \tilde{p}_g(t)
\end{equation*}
where
\begin{align*}
\tilde{p}_{g}(\eps^{-1}, Z) = \sum_{r>0, r|\beta} n_{g, \beta/r} \frac{1}{2}\sum_{r>0} \int^{\infty}_0 d\sigma \kappa_{\eps}(\sigma) e^{i  \sigma v_{\beta}} \frac{1}{r}(2\sin(r(\sigma-\pi)/2))^{2g-2}.
\end{align*} 
This explains the basic difficulty with signs for the unframed potential when $g > 1$. But in genus $g = 1$ it gives exactly what is needed: we have
\begin{equation*}
 \tilde{p}_{1}(\eps^{-1}, Z) = \sum_{r>0, r|\beta} n_{1, \beta/r} \frac{1}{2} \int^{\infty}_0 d\sigma \kappa_{\eps}(\sigma) \frac{1}{r}e^{i   \sigma v_{\beta}},
\end{equation*}
and proceeding as in the proof of Theorem \ref{mainThm} we find for all $j \geq 1$
\begin{align*}
&\mathcal{L}^{j}  \tilde{p}_1(\eps^{-1}, Z)\\&=\frac{1}{2} \sum_{r>0, r|\beta} n_{1, \beta/r}\Big( -2\sum^{\lfloor (j-1)/2 \rfloor}_{h = 0}(-1)^{j+h} \frac{(2h)!}{\eps^{2h+1}} \del^{j-2h -1}_{\sigma}\big( \frac{1}{r}e^{i\sigma   v_{\beta} }\big)|_{\sigma = 0}\\&   - (-1)^j \int^{\infty}_0 d\sigma\kappa_{\eps}(\sigma) \del^{j }_{\sigma}\big(  \frac{1}{r}e^{i\sigma  v_{\beta} } \big)\Big).
\end{align*}
Taking limits as $\eps \to 0$ we obtain
\begin{equation*}
\lim_{\eps\to 0} \tilde{p}_{1}(\eps^{-1}, Z) = \frac{1}{4} \sum_{r>0, r|\beta} n_{1, \beta/r} \frac{1}{r}e^{i  \sigma v_{\beta}}|_{\sigma = 0},
\end{equation*}
and for $j \geq 1$
\begin{align*}
&\lim_{\eps \to 0}\Big(\mathcal{L}^{j} \tilde{p}_1(\eps^{-1}, Z)\\
&-  \sum^{\lfloor (j-1)/2 \rfloor}_{h = 0}(-1)^{j+h} \frac{(2h)!}{\eps^{2h+1}} \del^{j-2h-1}_{\sigma}\big( \sum_{r | \beta } n_{1, \beta/r}\frac{1}{r}e^{i\sigma   v_{\beta} }\big)|_{\sigma = 0}\Big) \\
&=  - \frac{(-1)^j}{4} \del^{j}_{\sigma}\big(  \sum_{r>0, r|\beta}n_{1, \beta/r}\frac{1}{r}e^{i\sigma  v_{\beta} } \big)|_{\sigma = 0},
\end{align*} 
which shows how to extract the genus $g = 1$ contribution from flat sections of $\nabla$.

\end{document}